\documentclass{amsart}
\usepackage{amsmath, amsthm, amssymb}
\usepackage{verbatim}
\usepackage{tikz}

\renewcommand{\d}{\partial}

\newcommand{\veps}{\varepsilon}
\newcommand{\vphi}{\varphi}
\newcommand{\al}{\alpha}
\newcommand{\be}{\beta}

\newcommand{\de}{\delta}

\newcommand{\Om}{\Omega}
\newcommand{\De}{\Delta}

\newcommand{\cE}{\mathcal{E}}

\newcommand{\bC}{\mathbb{C}}

\newtheorem{thm}{Theorem}
\newtheorem{prop}[thm]{Proposition}
\newtheorem{lem}[thm]{Lemma}

\theoremstyle{definition}

\newtheorem{remark}[thm]{Remark}

\newtheorem{expl}[thm]{Example}

\numberwithin{thm}{section}
\numberwithin{equation}{section}

\renewcommand{\[}{\begin{equation}}
\renewcommand{\]}{\end{equation}}

\newcommand{\wed}{\wedge}

\newcommand{\wtd}{\widetilde}

\title[H\"older regularity of solutions to complex Hessian equations]{A remark on the H\"older regularity of solutions to the complex Hessian equation}

\author{S\l awomir Ko\l odziej and Ngoc Cuong Nguyen} 
\address{Faculty of Mathematics and Computer Science, Jagiellonian University, \L ojasiewicza 6, 30-348 Krak\'ow, Poland}
\email{slawomir.kolodziej@im.uj.edu.pl}
\address{Department of Mathematical Sciences, KAIST, 291 Daehak-ro, Yuseong-gu, Daejeon 34141, South Korea}
\email{cuongnn@kaist.ac.kr}

\begin{document} 

\begin{abstract} 
We prove that the Dirichlet problem for the complex Hessian equation has the H\"older continuous solution provided it has a subsolution with this property.
Compared to the previous result of Benali-Zeriahi and Charabati-Zeriahi
we remove the assumption on the finite total mass of the measure  on the right hand side.                 
\end{abstract}

\maketitle

\section{Introduction}

The theory of weak solutions  of the complex Monge-Amp\`ere equations started with the fundamental works of Bedford and Taylor \cite{BT76, BT82}. 
 It has found numerous profound applications in complex geometry (see e.g. the survey \cite{PS12}).
The H\"older continuous  solutions  are  interesting because such functions often appear in complex dynamic in higher dimension as potentials of  Green currents associated to  dominating holomorphic maps  \cite{DNS} and \cite{Si97}, as well as in approximation theory \cite{PP86}. 

Thanks to a series of works   \cite{BKPZ16}, \cite{BT76}, \cite{Cha15a, Cha15b}, \cite{GKZ08},  \cite{Ng18, Ng20}, \cite{hiep} and \cite{viet16} the H\"older regularity of the solution to the complex Monge-Amp\`ere equation  is rather well-understood.
 Notably, it was shown in \cite{BKPZ16} that if the measure on the right hand side has a density in $L^p$, $p>1$ and the boundary data  is H\"older continuous, then the solution is H\"older continuous. Such a regularity was also proved for the volume form of a generic smooth compact submanifold in $\bC^n$ in \cite{hiep} and \cite{viet16}. These measures are singular with respect to the Lebesgue measure.  A necessary and sufficient condition for the existence of a H\"older continuous solution via H\"older continuous subsolution was proved in \cite{Ng18, Ng20}, which was inspired by \cite{ko95}.

Recently, potential theory has been developed very successfully for the study of weak solutions to complex Hessian equations
(of which the  Monge-Amp\`ere equation is a special case)  both in domains in $\bC^n$ and on  compact K\"ahler manifolds by B\l ocki \cite{Bl05},  Dinew and the first author \cite{DK14, DK17}, and Lu \cite{Lu12, Lu13, Lu15}. 
The study of H\"older regularity of $m$-subharmonic solutions to the complex Hessian equation 
is recently  very active  \cite{BZ20},  \cite{Cha14},  \cite{ChZ20}, \cite{EZ}, \cite{KN20}, \cite{Li04}, \cite{Ng13} and \cite{V88}. Although the statements are often similar to the ones for the complex Monge-Amp\`ere equation,  the proofs are substantially different and usually the H\"older exponent of the solutions is
probably far from optimal.

We focus now on the H\"older regularity of $m$-subharmonic solutions to the complex Hessian equation assuming the existence of a H\"older continuous subsolution.
Let $\Om$ be a bounded domain in $\bC^n$. Let $1\leq m<n$ be an integer.  Let us denote by $\be = dd^c |z|^2$ the standard K\"ahler form and $SH_m(\Om)$ the set of all $m$-subharmonic functions in $\Om$ defined in \cite{Bl05}. For  $v \in SH_m(\Om)\cap L^\infty(\Om)$ its complex  Hessian measure is  given by
$$
	H_m(v) := (dd^cv)^m \wed \be^{n-m}.
$$
Let $\vphi \in SH_m(\Om) \cap C^{0,\al}(\bar\Om)$ with $0<\al \leq 1$ and let $\mu$ be a positive measure in $\Om$ satisfying
\[ \label{eq:subsolution} \mu \leq H_m(\vphi).
\]
To study the Dirichlet problem we assume that  $\Om$ is a bounded strictly  $m$- pseudoconvex domain in $\bC^n$. Let $\rho \in SH_m(\Om) \cap C^2(\bar\Om)$ be a strictly $m$-subharmonic defining function of $\Om$.  
By adding  $\vphi$ to an $m$-subharmonic envelope whose boundary value is equal to $-\vphi$ as in \cite[Theorem~4.7]{Ng13},  we obtain a new H\"older continuous subsolution 
(with possibly smaller H\"older exponent) 
with zero boundary value. Therefore, without loss of generality we may assume that
\[\label{eq:boundary-zero}
	\vphi =0 \quad \text{on }\d\Om.
\]
Given  the boundary data  $\phi \in C^{0,\al}(\d\Om)$ we wish to solve the Hessian equation
\[\label{eq:heq}
	u \in SH_m(\Om) \cap C^{0}(\bar\Om), \quad H_m(u) = \mu, \quad \text{and } u = \phi \quad\text{on }\d\Om.
\]
The main result of the paper is as follows.

\begin{thm}\label{thm:main} Let $\mu$ be the measure satisfying \eqref{eq:subsolution}. Then, the Dirichlet problem \eqref{eq:heq} is solvable and there exists $0<\al' \leq 1$ such that
\[\label{eq:holder-s} u \in C^{0,\al'} (\bar \Om).
\]
\end{thm}

This result was proved in \cite{KN20} for a compactly supported measure $\mu$ in $\Om$. Later, Benali-Zeriahi \cite[Theorem~B]{BZ20} (in a special case) and
Charabati and Zeriahi \cite[Theorem~4]{ChZ20} solved the Dirichlet problem  \eqref{eq:heq} $\&$ \eqref{eq:holder-s} for any  $\mu$ having finite total mass,  i.e., $\mu(\Om)<+\infty$. 
More precisely, it was a consequence of a general statement \cite[Theorem~3]{ChZ20} for the modulus of continuity under such an  assumption. In other words, they obtained the H\"older continuous subsolution theorem for finite mass measures. Here, we are able to remove the finite mass assumption. The method to prove the theorem is inspired by  \cite{BZ20} and \cite{KN20}. As in \cite{KN20} we do not invoke neither the viscosity solutions nor the rooftop envelopes. However, the H\"older exponent of the solution is worse than the one in \cite{BZ20} and \cite{ChZ20} (see also \cite{BZ20c}).

In general the complex Hessian measure of a H\"older continuous $m$-subharmonic function may have unbounded total mass.  Thus, for applications it is desirable to prove the H\"older continuous subsolution theorem for possibly unbounded total mass measures. Let us consider specific examples.

\begin{expl} Let $\Om$ be a strictly $m$-pseuconvex domain in $\bC^n$ and $\rho$ is a $C^2$ strictly $m$-subharmonic defining function of $\Om$ as above.
\begin{itemize}
\item 
[(a)] Consider the function $ - (-\rho)^\al$ for some $0<\al < 1$. Then, it is H\"older continuous on $\bar\Om$ and $m$-subharmonic in $\Om$. It is easy to see that for any $x \in \d\Om$ and $U_x$ a neighborhood of $x$,
$$
	\int_{\Om \cap U_x} H_m(-|\rho|^\al) = +\infty. 
$$
It follows that the measure $\mu = {\bf 1}_{U_x} \cdot H_m(-|\rho|^\al)$ has unbounded mass on $\Om$ and $\mu \leq H_m(-|\rho|^\al)$.
\item
[(b)] Take $G \subset \Om$ be such that $\mu := {\bf 1}_G \cdot H_m(-|\rho|^\al)$ satisfies $\mu( \Om) <+\infty$. Then, it is not clear how to find $\vphi \in \cE^0_{m}(\Om) \cap C^{0,\al}(\bar\Om)$ such that $\mu \leq H_m(\vphi)$, where $\cE^0_{m}(\Om)$ is the Cegrell class of bounded $m$-subharmonic functions whose boundary values are zero and whose total masses of complex Hessian measures are finite.
\end{itemize}
\end{expl}
 
 \begin{expl}
 Take $\vphi \in SH_m(\Om) \cap C^{0,\al}(\bar\Om)$ such that $\vphi \neq 0$ on $\d\Om$. Assume $\mu \leq H_m(\vphi)$ in $\Om$ and $\int_\Om H_m(\vphi) <+\infty$. Then, it is not clear if one can  find $\vphi'$  H\"older continuous $m$-subharmonic function such that $\vphi'=0$ on $\d\Om$ and $\int_\Om H_m(\vphi')<+\infty$.
\end{expl}

The H\"older regularity of solutions to complex Hessian equations or fully non-linear elliptic equations on compact K\"ahler manifold  is also  very active field of research see e.g., \cite{CY22}, \cite{DDGKPZ14}, \cite{DN16}, \cite{ko08} and \cite{WZ24}. It is worth to point out that the H\"older regularity of quasi $m$-subharmonic solutions to the equation on a compact K\"ahler manifold,
under similar hypothesis as the ones mentioned above, is widely open. There is a partial progress made by Cheng and Yu \cite{CY22} in case nonnegative holomorphic bisectional curvature of the metric.

\bigskip
\bigskip
{\em Acknowledgement.}   The paper has been completed while the second author was visiting  VIASM (Vietnam Institute for Advanced Study in Mathematics). He would like thank to the institute for the great hospitality and  excellent working conditions.  The first author is partially supported by  grant  no. 2021/41/B/ST1/01632 from the National Science Center, Poland.

\section{Hessian measures dominated by capacity}

In this section we consider the positive Radon measure $\mu$ which is dominated by the complex Hessian measure associated to a H\"older continuous $m$-subharmonic function $\vphi$ satisfying  \eqref{eq:subsolution} and \eqref{eq:boundary-zero}. Our goal is to prove that there exist uniform  constants $C, \al_0>0$ depending only on $\vphi$ and $\Om$ such that for every compact set $K \subset \Om$,
$$
	\mu (K) \leq C  \left[cap_m(K, \Om) \right]^{1+\al_0},
$$
where the $m$-capacity defined as in \cite{DK14}, namely
$$
	cap_m(K,\Om) = \sup \left\{\int_K H_m(w): w \in SH_m(\Om), \; -1\leq w\leq 0\right\}.
$$
This inequality was proved under the assumption of the compact support  of the measure in \cite{KN20}. Later it was proved without this extra hypothesis by Benali and Zeriahi \cite{BZ20} and a more general formulation was given by Charabati and Zeriahi \cite{ChZ20}, where they  employed the viscosity solution. Here we give an alternative proof with a smaller exponent  but it  does not invoke neither viscosity solutions nor the rooftop envelope.  This is done by slightly extending the H\"older continuous subsolution to a neighborhood of $\bar\Om$ using the defining function $\rho$. The cost is the growth of the H\"older norm of the new subsolution and we need to carefully analyze it. 

Since $\rho$ is a $C^2$  defining function of $\Om$,
we can assume that $\rho$ is a strictly $m$-subharmonic function in a neighborhood $\wtd\Om$ of $\bar \Om$.
Let us define for $\veps>0$
the sets
$$
	\Om_\veps = \{ z \in \Om : {\rm dist} (z, \d\Om) > \veps\},
$$
and 
$$
	D_\veps = \{z \in \Om : \rho(z) <-\veps\}.
$$
These two domains are comparable as, by the  Hopf lemma,  there exist constants $c_0,\veps_0>0$ such that 
\[\label{eq:constant-c0}
	\Om_\veps \subset D_{c_0\veps} \quad\text{for every }0< \veps \leq \veps_0.
\]
Let us consider the following extension of the subsolution $\vphi$
\[\label{eq:ext-subs}
	\psi := \max\left\{ \vphi -\veps, \frac{M\rho}{c_0\veps} \right\}, 
\]
where  $M = 1+\|\vphi\|_{L^\infty(\Om)}$. Clearly, $\psi$ is defined on $\wtd\Om$. Furthermore, 
$\psi = \vphi- \veps$ on $\Om_\veps$ and $\psi = M \rho/ (c_0\veps)$ near $\d\Om$. Thus, $\psi$ is $\al$-H\"older continuous and $m$-subharmonic in the neighborhood $\wtd\Om$ of $\bar\Om$ and
\[	
	H_m(\vphi) = H_m(\psi) \quad \text{on } \Om_\veps.
\]
Let us fix an $\veps_0>0$ depending only on $\Om$ such that for every $0<\veps\leq\veps_0$, 
\[\label{eq:range-epsilon}
	\bar\Om \subset \{ z \in \wtd\Om : {\rm dist}(z, \d \wtd\Om)>\veps\}.
\]
As $\vphi -\veps \leq \psi \leq 0$,
\[\label{eq:ext-norm-psi}
	\|\psi\|_{L^\infty (\bar\Om)} \leq M .
\]
Now we are ready  to give a slightly different proof of the following result due to Benali and Zeriahi \cite{BZ20}. 
\begin{thm} \label{thm:BZ}
There exist constants $A_1 >0$ and $\al_0>0$ such that for any compact set $K \subset \Om$ 
$$
	\int_K H_m(\vphi) \leq   A_1\left[cap_m(K)\right]^{1+\al_0}.
$$ 
Precisely,  $A_1 = C \|\vphi\|_{L^\infty(\Om)}^m \|\vphi\|_{C^{0,\al}(\bar\Om)}^m$ with $C$ a uniform constant depending only on $\Om$ and $\al_0$ which  is explicitly computable in terms of $\al, m$ and $n$.
\end{thm}

Following the idea in \cite{BZ20} we divide a compact $K$ into two parts $K \cap \Om_\veps$ and  $K\setminus \Om_\veps$. To deal with  the first part we need

\begin{lem}\label{lem:first-part} Fix $1<p<n/m$ and  $0<\veps \leq \veps_0$, where $\veps_0$ is given in \eqref{eq:constant-c0}. Let $E \subset \Om_{2\veps}$ be a Borel set. Then, 
$$
	\int_E H_m(\vphi) \leq A_2 \left(\veps^\al + \veps^{-2m} [cap_m(E)]^{p-1} \right)  cap_m(E),
$$
where $A_2=C  \|\vphi\|_{L^\infty(\Om)}^m \| \vphi\|_{C^{0,\al} (\bar\Om)}$.
\end{lem}

\begin{proof}
Without loss of generality we may assume that $E$ is a regular compact set. This means that  the relative extremal function
$$
	u_E (z) = \sup\{ v(z) : v\in SH_m(\Om): v \leq 0,\; v\leq -1 \text{ on } E\}
$$
is a continuous $m$-subharmonic function on $\bar\Om$. For the general $E$ one can use the fact that
$\mu$ is a Radon measure and $m$-capacity enjoys  the inner regularity property. Recall that $\psi$ is the extension defined in \eqref{eq:ext-subs}. Let us denote by $\psi_\veps = \psi *\chi_\veps$, where $\{\chi_\veps\}_{\veps>0}$ is the radial smoothing kernels, the  standard regularization of $\psi$ in $\wtd\Om$ for $0< \veps \leq \veps_0$. We have by \eqref{eq:range-epsilon} that the domain of $\psi_\veps$ contains $\bar \Om$. Hence, $\psi_\veps \in SH_m(\Om) \cap C^\infty(\bar\Om)$. Using the H\"older continuity of $\psi = \vphi -\veps$ we have 
$$
	0\leq \psi_\veps - \psi \leq c_1 \veps^\al:= \de \quad \text{on } \bar\Om_{2\veps},
$$
where $c_1 = \|\vphi\|_{C^{0,\al}(\bar\Om)}$.
The compactness of $E$ gives $u_E=0$ on $\d\Om$. Also $\psi_\veps \geq \psi$ on $\bar\Om$ by the subharmonicity of $\psi$. Consider the set  $E' := \{3\de u_E + \psi_\veps < \psi - 2\de\}$. Then,
$$	E \subset E' \subset \{u_E < -1/2\} \subset \subset \Om,
$$where  we used the fact $E\subset \Om_{2\veps}$ for the first inclusion and $\psi_\veps\geq \psi$ for the second one. In what follows we write for any Borel set $G \subset \Om$,
\[	cap_m(G) := cap_m(G, \Om)
\]
Clearly, $cap_m(E) \leq cap_m(E')$. On the other hand, by the comparison principle \cite[Theorem~1.4]{DK14},
\[\label{eq:compare}
\begin{aligned}
	cap_m(E') 
&\leq 	cap_m\left( \{u_E < -1/2\} \right) \\
&\leq  	2^m \int_{\{u_E <0\}} (dd^c u_E)^m \wed \beta^{n-m} \\
&=  		2^m cap_m(E). 
\end{aligned}\]
Next, notice from \eqref{eq:ext-norm-psi} that on $\bar\Om$,
$$
	dd^c \psi_\veps \leq \frac{C \|\psi\|_{L^\infty(\Om)} \be}{\veps^{2}}\leq \frac{C M \be}{\veps^{2}}, \quad \|\psi_\veps+ u_E\|_{L^\infty(\Om)} \leq M. 
$$
Since $E \subset \Om_\veps$, 
$\int_E H_m(\vphi) = \int_E H_m(\psi)$.
The comparison principle and the above inequalities give us that
$$\begin{aligned}
\int_{E'} H_m(\psi) 
&\leq		\int_{E'} H_m(3\de u_E + \psi_\veps) \\
&\leq		\int_{E'} 3\de H_m(u_E + \psi_\veps) + \int_{E'} H_m(\psi_\veps) \\
&\leq 	3\de M^m cap_m(E') + C M^{m}\veps^{-2m} \left[cap_m(E')\right]^p,
\end{aligned}$$
where we also used \cite[Proposition~2.1]{DK14} for the last inequality with $1<p < n/m$.
Combining with the equivalence \eqref{eq:compare} of $cap_m(E)$ and $cap_m(E')$ we get that for every Borel set $E \subset \Om_\veps$,
\[\label{eq:first-part}
	\int_E H_m(\vphi) \leq C M^m \left(c_1\veps^\al + \veps^{-2m} [cap_m(E)]^{p-1} \right)  cap_m(E).
\]
Since we may assume $c_1 = \|\vphi\|_{C^{0,\al}(\bar\Om)} \geq 1$, the desired inequality follows.
\end{proof}
The next lemma is due to Benali and Zeriahi \cite[Lemma~4.1]{BZ20} which takes care of the second part in the decomposition.

\begin{lem} Let $F \subset \Om \setminus \Om_{2\veps}$ be a compact set. Then,
\[\label{eq:second-part}
	\int_F H_m(\vphi) \leq  A_3\;\veps^{m\al} cap_m(F),
\]
where $A_3= 2^{m\al}\|\vphi\|_{C^{0,\al}(\bar\Om)}^m$.
\end{lem}

\begin{proof} Without loss of generality we can assume that $F$ is regular in the sense that the relative extremal function $u_F$ with respect to $\Om$ is continuous as in the proof of Lemma~\ref{lem:first-part}. Since $\vphi$ is H\"older continuous on $\bar\Om$ and $\vphi =0$ on $\d\Om$, we have for $z \in F$
$$
		-\vphi (z) \leq c_1 {\rm dist}(z, \d\Om)^\al < 2^\al c_1 \veps^\al =:\de,
$$
where $c_1= \|\vphi\|_{C^{0,\al} (\bar\Om)}$.
Consider the set 
$F' = \{\de u_F < \vphi \}$. Since $u_F=-1$ on $F$, we have $F \subset F'$. Note that $u_F = 0$ on $\d\Om$, hence $F' \subset \subset \Om$. By the comparison principle 
$$\begin{aligned}
	\int_{F'} H_m(\vphi) &\leq \int_{F'}H_m(\de u_F) \\&= \de^m \int_{F'}H_m (u_F) \\ &= \de^m cap_m(F).
\end{aligned}$$
The desired inequality follows.
\end{proof}

\begin{proof}[End of proof of Theorem~\ref{thm:BZ}]
Now for any compact set $K \subset \Om$ we write $K = (K \cap \Om_\veps) \cup (K  \setminus \Om_\veps).$  Let us denote by $E$ and $F$  the two sets of this union. Then, both of them have $m$-capacity  less than $cap_m(K)$. 
 It follows from \eqref{eq:first-part} and \eqref{eq:second-part} that
\[\label{eq:combination}\begin{aligned}
	\int_K H_m(\vphi) &=   \int_{E} H_m(\vphi) + \int_{F} H_m(\vphi) \\&\leq A_2 \left(\veps^\al + \veps^{-2m}\; [cap_m(E)]^{p-1}\right) cap_m(E) + A_3 \;\veps^{m\al} cap_m(F) \\&\leq A_4 \left(\veps^\al + \veps^{-2m} [cap_m(K)]^{p-1}+ \veps^{m\al}\right) cap_m(K),
\end{aligned}
\]
where $A_4 = \max\{A_2,A_3\}$.
Notice that this inequality holds for every $0<\veps\leq \veps_0$, where $\veps_0$ is fixed and defined in \eqref{eq:range-epsilon}.  

The proof of the theorem will follow by choosing a suitable $\veps>0$. To this end we define  $\al_1>0$ by 
$$
	p-1- 2m \al_1 = \al_1 \al \Leftrightarrow \al_1 = \frac{p-1}{2m+\al}.
$$
If $[cap_m(K)]^{\al_1} \geq \veps_0$, then by \eqref{eq:combination}
$$\begin{aligned}
	\int_K H_m(\vphi) 
&\leq 	A_4\left(\veps^\al_0 + \veps^{-2m}_0 cap_m(K)^{p-1}+ \veps^{m\al}_0\right) cap_m(K) \\
& \leq  	A_4 \left(\left[cap_m(K)\right]^{1+\al_1 \al} + \veps_0^{-2m} [cap_m (K)]^{p} + [cap_m(K)]^{1+ m\al_1  \al} \right).
\end{aligned}$$
So, we get the desired inequality in this case. Otherwise, if  $[cap_m (K)]^{\al_1} < \veps_0$, then we can choose $\veps = [cap_m(K)]^{\al_1}$. It follows from \eqref{eq:combination} that $\int_K H_m(\vphi)$ is less than
$$
	  A_4 \left( \left[cap_m(K)\right]^{1+\al_1\al} +  \left[cap_m (K)\right]^{p- 2m\al_1} + [cap_m(K)]^{1+ m\al_1  \al} \right).
$$
Therefore, the inequality 
$$
	\int_K H_m(\vphi) \leq A_1 \left[cap_m(K)\right]^{1+\al_0}
$$
holds for  $\al_0 = \al_1 \al>0$ and $A_1 = 3 A_4$. The proof of the theorem is completed.
\end{proof}

\section{The H\"older regularity of the solution}

Let us proceed to prove Theorem~\ref{thm:main}.
Thanks to the proof in \cite[Theorem~1.2]{KN20} and Theorem~\ref{thm:BZ} we can find $u \in SH_m(\Om) \cap C^0(\bar\Om)$ satisfying
\[\label{eq:cont-sol}
	(dd^c u)^m \wed \be^{n-m} = \mu, \quad u = \phi \quad\text{on } \d\Om.
\]
The continuous solution has been obtained by Charabati and Zeriahi \cite{ChZ20}  for a more general family of measures: those which are dominated by a complex Hessian measure associated to a continuous $m$-subharmonic function whose modulus of continuity satisfies the Dini-type condition.

Under the H\"older continuity of the boundary data $\phi \in C^{0,\al}(\d\Om)$ we wish to show that 
$$
	u \in C^{0,\al'}(\bar\Om) \quad \text{for some } \al'>0.
$$
Here, we may assume that the H\"older exponent of $\phi$ is the same as the one of $\vphi$, otherwise we can decrease $\al$.
We will see in the proof that the exponent $\al'$ is explicitly computable in terms of $\al, m$ and $n$. 
We follow closely the argument in \cite{Ng20} which deals with the complex Monge-Amp\`ere equation. Consider the sequence of sets $\Om_\veps$ with $\veps>0$ small. Recall that $M= 1+ \|\vphi\|_{L^\infty(\Om)}$ and denote
$$
	\wtd\vphi_\veps = \max\{\vphi -\veps, M \rho/\veps\}.
$$
This is similar to $\psi$ but we need to keep track of the dependence on $\veps$ more carefully this time. Notice that  $\wtd\vphi_\veps$ is $\al-$H\"older continuous and $m$-subharmonic in a neighborhood of $\bar\Om$, its value on the boundary $\d\Om$ is zero. Moreover, for $\veps>0$ small
\[\label{eq:norm-extension}
	\|\wtd\vphi_\veps\|_{L^\infty(\Om)} \leq \frac{A}{\veps},  \quad
	\|\wtd\vphi_\veps\|_{C^{0,\al}(\bar\Om)} \leq \frac{A}{\veps},
\]
where $A = M( \|\vphi\|_{C^{0,\al} (\bar\Om)} + \|\rho\|_{C^1(\bar\Om)})$.
Next, we can follow the lines of the proofs of \cite[Lemma~4.4, Remark~4.5, Corollary~4.6]{Ng20} to get the following two lemmata with the constant $C>0$ depending only on $\Om$ and $\rho$.

\begin{lem} \label{lem:mass-bound}We have for $k=1,...,m$
$$ 	\int_\Om (dd^c\wtd\vphi_\veps)^k \wed \be^{n-m+k}\leq \frac{CM^k}{\veps^k}.
$$
Moreover,  
$$	{\bf 1}_{D_\veps} \cdot \mu \leq H_m(\wtd\vphi_\veps) \quad\text{ on } D_\veps.
$$
\end{lem}

Since  the norms of  $\wtd\vphi_\veps$  satisfy \eqref{eq:norm-extension} we can apply Theorem~\ref{thm:BZ} for this H\"older continuous subsolution to get

\begin{lem}
\label{lem:hess-cap} Let $\al_0$ be defined in Theorem~\ref{thm:BZ}. Then, 
for every compact set $K \subset \Om$, 
$$
	\int_K H_m(\wtd\vphi_\veps) \leq \frac{C A^{2m}}{\veps^{2m}} [cap_m (K)]^{1+\al_0}.
$$
\end{lem}

We are in the position to state an analogue of \cite[Proposition~4.7]{Ng20}  for the complex Hessian equation.

\begin{prop} \label{prop:stability}
Let $u$ be the solution in \eqref{eq:cont-sol}. Let $v \in SH_m\cap L^\infty(\Om)$ be such that $v = u$ on $\Om \setminus \Om_\veps$. Then, there is $0<\al_2<1$ such that
$$
	\sup_{\Om} (v-u) \leq \frac{C}{\veps^{2m}} \left(\int_\Om\max\{v-u, 0\} d\mu\right)^{\al_2},
$$
where $C>0$ is a uniform constant independent of $\veps$.
\end{prop}

\begin{proof} 
Without loss of generality we may assume that $ \sup_{\Omega} (v-u)>0$. Set \[s_0 := \inf_{\Omega} (u-v).\] We know that
for $0<s < |s_0|$,
\[
	U(s):=\{u< v + s_0 +s\} \subset \subset \Omega_\veps.
\]
The only difference from \cite[Proposition~2.5]{KN20} is that now the constant $C/\veps^{m}$ appeared on the right hand side. 
As shown in \cite{Ng20} for the Monge-Amp\`ere equation  the proof of the proposition will follow once we have the following inequality.

\begin{lem}\label{lem:capacity-growth}
Let $\al_0>0$ be as in Theorem~\ref{thm:BZ}. Then,  for every $0< s, t < \frac{|s_0|}{2}$,
\[
	t^m cap_m(U(s)) \leq \frac{C(\al_0)}{\veps^{2m}} \left[ cap_m(U(s+t)) \right]^{1+\al_0}.
\]
\end{lem}
\begin{proof}[Proof of Lemma~\ref{lem:capacity-growth}]
Let $0 \leq w \leq 1$ be a $m$-subharmonic function in $\Omega$. By the comparison principle \cite[Theorem~1.4]{DK14},
$$
 \int_{\{u<v+s_0 +s+tw\}} 	H_m(v+tw) \leq		\int_{\{u<v+s_0 +s+tw\}} 	H_m(u),
$$
Since $U(s) \subset \{u<v+s_0+s + tw\} \subset U(s+t)$, it follows that
\[\notag\begin{aligned}
\int_{U(s)} H_m(v+tw) 
&\leq		\int_{\{u<v+s_0 +s+tw\}} 	H_m(v+tw) \\
&\leq		\int_{\{u<v+s_0 +s+tw\}} 	H_m(u),  \\
&\leq		\int_{U(s+t)} d\mu.
\end{aligned}\]
As $t^m \int_{U(s)} H_m(w) \leq \int_{U(s)} H_m(v+ tw )$, taking supremum over all  such  $w$ we get 
\[\label{eq:cap-decay-a}
	t^m cap_m(U(s)) \leq \int_{U(s+t)} d\mu.
\]
On the other hand, by Lemma~\ref{lem:mass-bound}  $${\bf 1}_{D_{c_0\veps}} \cdot d\mu \leq H_m(\wtd\vphi_{c_0\veps}) \quad \text{in } \Om$$ as two measures, where $c_0$ is the constant defined in \eqref{eq:constant-c0}.
Combining this with $U(s+t) \subset \Omega_\veps \subset D_{c_0\veps}$ we have 
\[\label{eq:cap-decay-b}\begin{aligned}
	\int_{U(s+t)} d\mu 
&\leq		 \int_{U(s+t)} H_m(\wtd\vphi_{c_0\veps})\\
&\leq		\frac{C(\al_0)}{(c_0\veps)^{2m}} \left[ cap_m(U(s+t)) \right]^{1+\al_0},
\end{aligned}\]
where the last inequality follows from Lemma~\ref{lem:hess-cap}. Thus, the proof of the lemma follows from \eqref{eq:cap-decay-a} and \eqref{eq:cap-decay-b}.
\end{proof}
Then, the rest of the proof of the proposition is the same as in \cite[Theorem~1.1]{GKZ08} with obvious change of notations.
\end{proof}

An important step is to derive the $L^1 - L^1$ (with respect to the measure $\mu$ and the Lebesgue measure) stability estimate of solutions. The one stated below is an analogue of \cite[Theorem~4.9]{Ng20} for the complex Hessian equation. For the sake of completeness we give the detailed proof.

\begin{thm} \label{thm:l1-l1} 
Let $u$ be the solution in \eqref{eq:cont-sol}. Let $v \in SH_m\cap L^\infty(\Om)$ be such that $v = u$ on $\Om \setminus \Om_\veps$. Then, there is $0<\al_3<1$ such that 
$$
	\int_\Om |v-u| d\mu \leq \frac{C}{\veps^{m+1}} \left(\int_\Om |v-u| dV_{2n}\right)^{\al_3},
$$
where $C>0$ is a uniform constant independent of $\veps$.
\end{thm}

\begin{proof} Put 
\[
	S_{k,\veps} := (dd^c\wtd\vphi_\veps)^k \wed \beta^{n-k}, \quad k =0,...,m.
\]
where $\wtd\vphi_\veps= \max\{\vphi - \veps, M\rho/\veps\}$. Since $\mu \leq S_{m,\veps}$ on $\Omega_\veps$, it is enough to show that there is $0<\tau \leq 1$ satisfying
\[\label{eq:thm-l1-l1-2}
	\int_\Omega (v-u) S_{m,\veps}  \leq \frac{C}{\veps^{m+1}} \|v-u\|_1^{\tau}.
\]
for $v\geq u$ on $\Omega$, where here and below we denote by
$\| \cdot \|_1$  the $L^1$-norm with respect to the Lebesgue measure.
 (In the general case we use the  identity
\[\notag
	|v-u| = (\max\{v,u\} -u) + (\max\{v,u\} -v)
\]
and apply twice the inequality \eqref{eq:thm-l1-l1-2} to get the theorem.)

Now we can use the inductive arguments in $0\leq k \leq m$. Clearly it holds for $k=0$ with $\al_2=1$. Assume the inequality holds for $k<m-1$. Let us denote $\{\chi_t\}_{t>0}$ the family of standard radial smoothing kernels. To pass from the  $k$-th step to the step number $(k+1) \leq m$ 
we need the following inequality (with the notation $S_\veps:= (dd^c \wtd\vphi_\veps)^{k} \wed \beta^{n-k-1}$)
\[\label{eq:ind-est}\begin{aligned}
	\int_\Omega (v-u) dd^c \wtd\vphi_\veps \wed S_\veps 
&\leq 	\left| \int_\Omega (v-u) dd^c \wtd\vphi_\veps * \chi_t \wed S_\veps \right| \\
&\quad	+\left| \int_\Omega (v-u) dd^c (\wtd\vphi_\veps * \chi_t -\vphi_\veps) \wed S_\veps\right| \\
&=: I_{1,\veps}+ I_{2,\veps}.
\end{aligned}\]
Notice that  $\wtd\vphi_\veps$ 
is well-defined on the neighborhood $\wtd\Om$ of $\bar\Omega$, so is $\wtd\vphi_\veps *\chi_t$ for $t>0$ small. Moreover, 
$	
\|\wtd\vphi_\veps\|_{C^{0,\alpha}(\wtd\Om)} \leq C/\veps.
$
Since for $t>0$,
\[\label{eq:observe-holder-b}
	\left|\frac{\d^2 \wtd\vphi_\veps * \chi_t}{\d z_j \d \bar z_k} (z)\right| \leq \frac{C \|\wtd\vphi_\veps\|_{L^\infty(\Om)}}{ t^2} \leq \frac{C (1 + \|\vphi\|_{L^\infty(\Om)})}{t^2},
\]
and by  the induction hypothesis at  $k$-th step, there exists $0<\tau_k \leq 1$ such that
$$
\int_\Omega (v-u) S_\veps \wed \beta \leq \frac{C}{\veps^{k+1}} \|v-u\|_1^{\tau_k},
$$
we conclude that
\[\label{eq:ind-est-a}
\begin{aligned} I_{1,\veps}
&\leq 	\frac{C(1 + \|\vphi\|_{L^\infty(\Om)})}{ t^2} \int_\Omega (v-u) S_\veps \wed \beta \\
&\leq 	\frac{C(1 + \|\vphi\|_{L^\infty(\Om)})}{\veps^{k+1}\, t^2} \|v-u\|_1^{\tau_k}.
\end{aligned}\]
By integration by parts, using $u=v$ on $\Omega\setminus \Omega_\veps$, and 
\[\notag \label{eq:observe-holder-a}
\begin{aligned}
	\left|\wtd\vphi_\veps * \chi_t (z) - \wtd\vphi_\veps(z) \right|
&	\leq \frac{Ct^{\alpha}}{\veps},
\end{aligned}\]
it follows that 
\[\label{eq:i2e-a}\begin{aligned}
	I_{2,\veps} 
&\leq		\frac{C t^\alpha}{\veps} \int_{\Omega_\veps} (dd^c v+ dd^cu) \wed S_\veps. \\
\end{aligned}\]
At this point  the total mass of the right hand side may go to infinity as  $\veps \to 0^+$. However, we can control it uniformly in terms of a negative exponent of $\veps$. The argument is inspired by  \cite[Lemma~4.3]{Ng20}.  Namely,  we obtain
\[\label{eq:i2e-b}
\int_{\Omega_\veps} (dd^c u + dd^c v) \wed S_\veps  \leq  \frac{C\|u+v\|_{L^\infty(\Om)} (1+\|\vphi\|_{L^\infty(\Om)})^k}{\veps^{k+1}} .
\]
Indeed, we first have 
\[\notag\begin{aligned}
&\int_{\Omega_\veps} dd^c (u+v) \wed (dd^c \wtd\vphi_\veps)^k\wed \beta^{n- k-1}  \\
&\leq		\frac{2}{\veps'}\int_{\Omega} \left(\max\{\rho, - \veps'/2\} - \rho\right) \wed dd^c (u+v) \wed (dd^c \wtd\vphi_\veps)^k\wed \beta^{n-k-1} \\
&\leq \frac{C}{\veps} \|u+v\|_{L^\infty(\Om)}  \int_\Omega (dd^c\rho)\wed (dd^c \wtd\vphi_\veps)^k\wed \beta^{n-k-1},
\end{aligned}\]
where $\veps' = c_0\veps$ with $c_0$ defined by \eqref{eq:constant-c0}. 
The desired inequality \eqref{eq:i2e-b} follows from Lemma~\ref{lem:mass-bound}. 
Now, combining \eqref{eq:i2e-a} and \eqref{eq:i2e-b}  we get that
\[\label{eq:ind-est-b}
I_{2,\veps}  \leq		\frac{C t^\alpha}{\veps^{k+2}}.
\]
Next, it is easy to see (from Lemma~\ref{lem:mass-bound}) that 
\[\notag
	\int_{\Omega} (v-u) S_{m,\veps} \leq \frac{C \|u\|_{L^\infty(\Om)} (1+ \|\vphi\|_{L^\infty(\Om)})^m}{\veps^{m}}. 
\]
Therefore, we can assume that $0< \|v-u\|_1 < 0.01$.
Thanks to \eqref{eq:ind-est-a} and \eqref{eq:ind-est-b} we have
\[\notag
	\int_\Omega (v-u) dd^c \wtd\vphi_\veps\wed S_\veps \leq \frac{C}{\veps^{k+1} \, t^2} \|v-u\|_1^{\tau_k} + \frac{C t^\alpha}{\veps^{k+2}}.
\]
If we choose 
$
	t = \|v-u\|_1^\frac{\tau_k}{3}, 
	\quad \tau_{k+1} = \frac{\alpha\tau_k}{3},
$
then
\[\notag
	\int_\Omega (v-u) S_\veps \wed dd^c \wtd\vphi_\veps \leq \frac{C}{\veps^{k+2}} \|v-u\|_1^{\tau_{k+1}}.
\]
Thus, the induction argument is completed, and the theorem follows. 
\end{proof}

The next result is from \cite[Lemma~4.10]{Ng20} and its proof used only the subharmonicity. Therefore, it is also valid for $m$-subharmonic functions.
\begin{lem}  \label{lem:lap-est} Let $u$ be the solution in \eqref{eq:cont-sol} and define 
 for $z\in \Omega_\delta$, $$\hat u_\delta(z):= \frac{1}{\sigma_{2n}\delta^{2n}} \int_{|\zeta| \leq \delta} u(z+\zeta) dV_{2n}(\zeta),$$
where $\sigma_{2n}$ is the volume of the unit ball.  Then,  for $\delta>0$ small,
\[
	\int_{\Omega_\delta} |\hat u_\delta -u | dV_{2n} \leq C \delta.
\]
\end{lem}

\begin{proof} Since $\Om_\de \subset D_{c_0\de}$, we have
\[\begin{aligned}
	\int_{\Omega_\delta} dd^c u(z) \wed \be^{n-1} &\leq \frac{2}{c_0\delta} (\max\{\rho, -c_\de/2\} - \rho) dd^c u \wed \be^{n-1} \\&\leq \frac{C \|u\|_{L^\infty(\Om)}}{\de} \int_\Om dd^c\rho\wed \be^{n-1}.
\end{aligned}\]
Therefore, using the classical Jensen formula for subharmonic functions (see e.g., \cite[Lemma~4.3]{GKZ08}),
\[\begin{aligned}
	\int_{\Omega_\delta} |\hat u_\delta -u| dV_{2n} &\leq \int_{\Omega_{2\delta}} |\hat u_\delta -u| dV_{2n} + \|u\|_{L^\infty(\Om)} \int_{\Omega_\delta\setminus \Omega_{2\delta}} dV_{2n}  \\&\leq C \de^2 \int_{\Om_\de} \De u + c_6 \de \\&\leq C\de. 
\end{aligned}\]
This is the required inequality.
\end{proof}

We are ready to prove the H\"older continuity of the solution $u$ in \eqref{eq:cont-sol}.

\begin{proof}[End of Proof of  Theorem~\ref{thm:main}] Let us fix $\delta$ such that $0< \delta < \delta_0$ small and let $\veps$ be such that $ \delta \leq \veps < \delta_{0}$ which is to be determined later. Recall that
\[
	\Omega_\delta:= \{z\in \Omega: dist(z,\d\Omega) >\delta\};
\]
and for $z\in \Omega_\delta$ we define
\[\begin{aligned}
&	u_\delta(z) := \sup_{|\zeta| \leq \delta} u(z+\zeta).
\end{aligned}\]

Directly from definitions and the H\"older continuity of the boundary data as in \cite[Lemma~4.1]{Ng20} we have the following inequalities.

\begin{lem} \label{lem:boundary-holder}
We have for $z\in \bar\Omega_{\delta}\setminus\Omega_{\veps}$,
\[	u_\delta(z) \leq u(z) + c_5 \veps^\alpha,
\]
where $c_5= \|\vphi\|_{C^{0,\al}(\bar\Om)} + \|\phi\|_{C^{0,\al}(\bar\Om)}$.
In particular, 
\[
	\sup_{\Omega_\delta} (\hat u_\delta - u) \leq \sup_{\Omega_\veps} (\hat u_\delta - u) + c_5 \veps^\alpha.
\]
\end{lem}

\begin{remark} It is important to keep in mind that the uniform constants $C, c_i>0$ which appear in the lemma, and several times below are independent of  $\delta$ and  $\veps$. 
\end{remark}

Thanks to  Lemma~\ref{lem:boundary-holder} and $\hat u_\delta \leq u_\delta$ we have $\hat u_\delta - c_5\veps^\alpha \leq u$ on $\d\Omega_\veps$. Therefore, the function
\[ \label{eq:extend-solution}
\tilde u := 
\begin{cases} 
	\max\{\hat u_\delta - c_5 \veps^\alpha, u\} \quad 
	&\mbox{ on } \Omega_{\veps},\\
	u  \quad &\mbox{ on } \Omega\setminus \Omega_{\veps},
\end{cases}
\]
belongs to $SH_m(\Omega)\cap C^0(\bar\Omega)$. Notice that $\tilde u \geq u$ in $\Omega$, and
\[
	 \tilde u = u \quad\mbox{ on } \Omega\setminus \Omega_\veps.
\]
Again, by the second part of Lemma~\ref{lem:boundary-holder}  we have that
\[\label{eq:holder-eq1}
\begin{aligned}
	\sup_{\Omega_\delta}(\hat u_\delta -u) 
&\leq 	\sup_{\Omega_\veps} (\hat u_\delta -u) + c_5 \veps^\alpha \\
&\leq		\sup_{\Omega} (\tilde u - u) + 2c_5 \veps^\alpha.
\end{aligned}\]
By the stability estimate (Proposition~\ref{prop:stability}) there exists $0<\alpha_2 \leq 1$ such that
\[\label{eq:holder-eq2}\begin{aligned}
	\sup_{\Omega} (\tilde u - u) 
&\leq 	\frac{C}{\veps^{2m}} \left(\int_{\Omega} \max\{\tilde u - u,0\} d\mu\right)^{\alpha_2} \\
&\leq		\frac{C}{\veps^{2m}} \left(\int_{\Omega} |\tilde u - u|  d\mu\right)^{\alpha_2},	
\end{aligned}\]
where we used the fact that $\tilde u = u$ outside $\Omega_\veps$.
Using Theorem~\ref{thm:l1-l1}, there is $0<\alpha_3\leq 1$ such that
\[\label{eq:holder-eq3}
\begin{aligned}
	\sup_{\Omega} (\tilde u -u) 
&\leq 	\frac{C}{\veps^{2m+(m+1)\alpha_2}}  \left(\int_{\Omega} |\tilde u - u|  dV_{2n}\right)^{\alpha_2\alpha_3} \\
&\leq		\frac{C}{\veps^{3m+1}}  \left(\int_{\Omega_\delta} |\hat u_\delta - u|  dV_{2n}\right)^{\alpha_2\alpha_3}, 
\end{aligned}\]
where we used $0\leq \tilde u - u \leq {\bf 1}_{\Omega_\veps} \cdot (\hat u_\delta -u)$ and $\Omega_\veps \subset \Omega_{\delta}$ for the second inequality. It follows from \eqref{eq:holder-eq1}, \eqref{eq:holder-eq3} and Lemma~\ref{lem:lap-est} that
\[
	\sup_{\Omega_\delta} (\hat u_\delta - u) \leq C \veps^\alpha + \frac{C \delta^{\alpha_2\alpha_3}}{\veps^{3m+1}}.
\]
Now, we choose $\alpha_4 = \alpha\alpha_2\alpha_3/(3m+1+\alpha)$ and 
$$ \veps = \delta^{\frac{\alpha_2\alpha_3 }{3m+1 + \alpha}}.
$$
Then, 
$
	\sup_{\Omega_\delta} (\hat u_\delta - u) \leq  C \delta^{\alpha_4}.
$
Finally, thanks to \cite[Lemma~4.2]{GKZ08}  (see also \cite{Z21}) we infer that 
$
\sup_{\Omega_\delta} (u_\delta - u) \leq C \delta^{\alpha_4}.
$
The proof of the theorem is finished.
\end{proof}

\end{document}